\documentclass[11pt]{amsart}
\usepackage{amscd,amssymb}
\usepackage{amsthm,amsmath,amssymb}
\usepackage[matrix,arrow]{xy}

\theoremstyle{definition}

\newtheorem{definition}[equation]{Definition}
\newtheorem{theorem}[equation]{Theorem}
\newtheorem{lemma}[equation]{Lemma}
\newtheorem{corollary}[equation]{Corollary}
\newtheorem{conjecture}[equation]{Conjecture}
\newtheorem{question}[equation]{Question}
\newtheorem*{question*}{Question}

\newtheorem*{problem*}{Problem}

\theoremstyle{remark}
\newtheorem{remark}[equation]{Remark}

\def\C {\mathbb{C}}
\def\Z {\mathbb{Z}}

\def\P {\mathbb{P}}

\def\OO {\mathcal{O}}
\def\GL {\mathrm{GL}}

\def\SL {\mathrm{SL}}

\def\A {\mathrm{A}}
\def\SS {\mathrm{S}}

\def\lct {\mathrm{lct}}
\def\Sym {\mathrm{Sym}}

\def\le {\leqslant}
\def\ge {\geqslant}

\author{Ivan Cheltsov and Constantin Shramov}

\title{Nine-dimensional exceptional quotient singularities exist}

\address{University of Edinburgh, Edinburgh EH9 3JZ, UK, \texttt{I.Cheltsov@ed.ac.uk}}

\address{Steklov Institute of Mathematics, Moscow 119991, Russia, \texttt{shramov@mccme.ru}}

\address{Laboratory of Algebraic Geometry, GU-HSE, 7 Vavilova street, Moscow 117312, Russia}%

\thanks{We assume that all varieties are projective, normal, and defined over $\mathbb{C}$.}

\begin{document}

\begin{abstract}
We prove that nine-dimensional exceptional quotient singularities
exist.
\end{abstract}

\maketitle

The~Fubini--Studi metric on $\mathbb{P}^{n}$ is known to be
K\"ahler--Einstein. Moreover, it follows from~\cite{BandoMabuchi}
that every K\"ahler--Einstein metric on $\mathbb{P}^{n}$ is a pull
back of the~Fubini--Studi metric (possibly multiplied by a
positive real constant) via some automorphism of $\P^n$. However,
there are plenty of non-K\"ahler--Einstein metrics
on~$\mathbb{P}^{n}$ whose K\"ahler forms lie in
$\mathrm{c}_1(\P^n)$. Let $g=g_{i\overline{j}}$ be such a metric
with a K\"ahler form $\omega$. Then one can try to obtain the
K\"ahler--Einstein metric on $\P^n$ out of the metric $g$ by
taking the normalized K\"ahler--Ricci iterations defined by
\begin{equation}
\label{equation:iteration} \left\{\aligned
&\omega_{i-1}=\mathrm{Ric}\big(\omega_{i}\big),\\
&\omega_{0}=\omega,\\
\endaligned
\right.
\end{equation}
where $\omega_{i}$ is a~K\"ahler form such that
$\omega_{i}\in\mathrm{c}_1(\P^n)$. Indeed, it follows from
\cite{Yau78} that the~solution~$\omega_{i}$ to
$(\ref{equation:iteration})$ exists for every $i\geqslant 1$.
However, it is not clear that any solution to
$(\ref{equation:iteration})$ converges to the K\"ahler form of the
K\"ahler--Einstein metric on $\P^n$ in the sense of
Cheeger--Gromov (see \cite{Rub08}). Nevertheless, this is known to
be true under an additional assumption that we are going to
describe.

Let $\bar{G}\subset\mathrm{Aut}(\P^n)$ be a~finite subgroup.
Suppose, in addition, that the metric $g$ is $\bar{G}$-invariant.
Let $\alpha_{\bar{G}}(\P^n)$ be the $\bar{G}$-invariant
$\alpha$-invariant of Tian of $\P^n$ that is introduced in
\cite{Ti87}.

\begin{theorem}[{\cite{Rub08}}]
\label{theorem:Yanir} If $\alpha_{\bar{G}}(\P^n)>1$, then
any~solution to $(\ref{equation:iteration})$ converges to the
K\"ahler form of the K\"ahler--Einstein metric on $\P^n$ in
$C^{\infty}(X)$-topology.
\end{theorem}

It should be mentioned that the original result by Rubinstein
proved in \cite{Rub08} is much stronger that
Theorem~\ref{theorem:Yanir} and valid for any smooth complex
manifold with a positive first Chern class. Nevertheless, even in
the simplest possible case of $\P^n$, the assertion of
Theorem~\ref{theorem:Yanir} is still very not obvious. Thus, it is
natural to ask the following

\begin{question}[Rubinstein]
\label{question:Yanir} Is there a~finite subgroup
$\bar{G}\subset\mathrm{Aut}(\mathbb{P}^{n})$ such that
$\alpha_{\bar{G}}(\mathbb{P}^{n})>1$?
\end{question}

It came as a surprise that Question~\ref{question:Yanir} is
strongly related to the notion of exceptional singularity that was
introduced by Shokurov in \cite{Sho93}. Let us recall this notion.

\begin{definition}[Shokurov]
\label{definition:exceptional} Let $(V\ni O)$ be a~germ of
Kawamata log terminal singularity. Then $(V\ni O)$ is~said to be
\emph{exceptional} if for every effective $\mathbb{Q}$-divisor
$D_{V}$ on the~variety $V$ such that $(V,D_{V})$ is log canonical
and for every resolution of singularities $\pi\colon U\to V$ there
exists at most one $\pi$-exceptional divisor $E\subset
U$~such~that $a(V,D_{V},E)=-1$, where the rational number
$a(V,D_{V},E)$ can be defined through the equivalence
$$
K_{U}+D_{U}\sim_{\mathbb{Q}}\pi^{*}\Big(K_{V}+D_{V}\Big)+
\sum a\big(V,D_{V},F\big)F.%
$$
The sum above is taken over all $f$-exceptional divisors, and
$D_{U}$ is the~proper transform of the~divisor $D_{V}$ on
the~variety $U$.
\end{definition}

One can show that exceptional Kawamata log terminal singularities
are straightforward generalizations of the Du Val singularities of
type $\mathbb{E}_{6}$, $\mathbb{E}_{7}$ and $\mathbb{E}_{8}$ (see
\cite[Example~5.2.3]{Sho93}), which partially justifies the word
``exceptional'' in Definition~\ref{definition:exceptional}. It
follows from \cite[Theorem~1.16]{ChSh09} that exceptional Kawamata
log terminal singularities exist in every dimension. Surprisingly,
Question~\ref{question:Yanir} is \emph{almost} equivalent to the
following

\begin{question}
\label{question:exceptional} Are there exceptional \emph{quotient}
singularities of dimension $n+1$?
\end{question}

Recall that quotient singularities are always Kawamata
log~terminal. So Question~\ref{question:exceptional} fits well to
Definition~\ref{definition:exceptional}. It follows
from~\cite{Sho93}, \cite{MarPr99}, \cite{ChSh09}, and
\cite{ChSh10} that the answers to both
Questions~\ref{question:Yanir}~and~\ref{question:exceptional} are
positive for every $n\leqslant 5$ (see
Theorems~\ref{theorem:Vanya-Kostya-invariants} and
\ref{theorem:dime-6}). Moreover, it follows from \cite{ChSh10}
that the answers to both Questions~\ref{question:Yanir} and
\ref{question:exceptional} are ``surprisingly'' negative for
$n=6$. The purpose of this paper is to show that the answers to
both Questions~\ref{question:Yanir} and \ref{question:exceptional}
are again positive for $n=8$ by proving the following

\begin{theorem}
\label{theorem:main} Let $G$ be a finite subgroup in $\SL_9(\C)$
such that $G\cong 3^{1+4}:\mathrm{Sp}_{4}(3)$ (see \cite{Atlas}
for notation), let
$\phi\colon\SL_{9}(\mathbb{C})\to\mathrm{Aut}(\mathbb{P}^{8})$ be
the~natural projection. Put $\bar{G}=\phi(G)$. Then $4/3\geqslant
\alpha_{\bar{G}}(\mathbb{P}^{8})\geqslant 10/9$ and the
singularity $\C^9/G$ is exceptional.
\end{theorem}

How to compute $\alpha_{\bar{G}}(\mathbb{P}^{n})$? How to show
that a given quotient singularity is exceptional? How
Questions~\ref{question:Yanir} and \ref{question:exceptional} are
related? How to prove Theorem~\ref{theorem:main}? What are the
expected answers to Questions~\ref{question:Yanir} and
\ref{question:exceptional} for $n=7$ and $n\geqslant 9$? Let us
give partial answers to these questions.

Let
$\phi\colon\GL_{n+1}(\mathbb{C})\to\mathrm{Aut}(\mathbb{P}^{n})$
be the~natural projection. Then there exists a~finite subgroup~$G$
in $\GL_{n+1}(\mathbb{C})$ such that $\phi(G)=\bar{G}$. Put
$$
\mathrm{lct}\Big(\mathbb{P}^{n},\bar{G}\Big)=\mathrm{sup}\left\{\lambda\in\mathbb{Q}\ \left|%
\aligned
&\text{the~log pair}\ \left(\mathbb{P}^{n}, \lambda D\right)\ \text{has log canonical singularities}\\
&\text{for every $\bar{G}$-invariant effective $\mathbb{Q}$-divisor}\ D\sim_{\mathbb{Q}} -K_{\mathbb{P}^{n}}\\
\endaligned\right.\right\}.%
$$

\begin{theorem}[{see e.\,g. \cite[Theorem~A.3]{ChSh08c}}]
\label{theorem:alpha} One has
$\mathrm{lct}(\mathbb{P}^{n},\bar{G})=\alpha_{\bar{G}}(\mathbb{P}^{n})$.
\end{theorem}

The number $\mathrm{lct}(\mathbb{P}^{n},\bar{G})$ is usually
called $\bar{G}$-equivariant \emph{global log canonical threshold}
of $\mathbb{P}^{n}$. Despite the fact that
$\mathrm{lct}(\mathbb{P}^{n},\bar{G})=\alpha_{\bar{G}}(\mathbb{P}^{n})$,
we still prefer to work with the number
$\mathrm{lct}(\mathbb{P}^{n},\bar{G})$ throughout this paper,
because it is easier to handle
than~$\alpha_{\bar{G}}(\mathbb{P}^{n})$. For example, it follows
immediately from the definition~of the number
$\mathrm{lct}(\mathbb{P}^{n},\bar{G})$ that
$\mathrm{lct}(\mathbb{P}^{n},\bar{G})\leqslant d/(n+1)$ if
the~group~$G$ has a~semi-invariant of degree~$d$ (a semi-invariant
of the group $G$ is a polynomial whose zeroes define a
$\bar{G}$-invariant hypersurface in $\mathbb{P}^{n}$).

Recall that an element $g\in G$ is called a \emph{reflection} (or
sometimes a \emph{quasi-reflection}) if there is a hyperplane in
$\mathbb{P}^{n}$ that is pointwise fixed by $\phi(g)$. To answer
Question~\ref{question:exceptional} one can always assume that the
group $G$ does not contain reflections (cf.
\cite[Remark~1.16]{ChSh10}). On the other hand, one can easily
check that there exists a~finite subgroup
$G^{\prime}\subset\SL_{n+1}(\mathbb{C})$ such~that
$\phi(G^{\prime})=\bar{G}$. So to answer
Question~\ref{question:Yanir} one can also assume that
$G\subset\SL_{n+1}(\mathbb{C})$, which implies, in particular,
that the group~$G$ does not contain reflections. Moreover, if the
group $G$ does not contain reflections, then the singularity
$\mathbb{C}^{n+1}\slash G$ is exceptional if and only if
the~singularity $\mathbb{C}^{n+1}\slash G^{\prime}$ is exceptional
thanks to the following

\begin{theorem}[{\cite[Theorem~3.17]{ChSh09}}]
\label{theorem:criterion} Let $G$ be a~finite subgroup in
$\GL_{n+1}(\mathbb{C})$ that does not contain reflections. Then
the singularity~$\C^{n+1}\slash G$ is exceptional if and only if
for any $\bar{G}$-invariant effective $\mathbb{Q}$-divisor $D$ on
$\mathbb{P}^{n}$ such that $D\sim_{\mathbb{Q}}
-K_{\mathbb{P}^{n}}$, the~log pair $(\mathbb{P}^{n}, D)$ is
Kawamata log terminal.
\end{theorem}

\begin{corollary}
\label{corollary:criterion} Let $G$ be a~finite subgroup in
$\GL_{n+1}(\mathbb{C})$ that does not contain reflections. Then
\begin{itemize}
\item the~singularity $\C^{n+1}/G$ is exceptional if $\mathrm{lct}(\mathbb{P}^{n},\bar{G})>1$,%
\item the~singularity $\C^{n+1}/G$ is not exceptional if either $\mathrm{lct}(\mathbb{P}^{n},\bar{G})<1$,%
\item the~singularity $\C^{n+1}/G$ is not exceptional if $G$ has a~semi-invariant of degree~at~most~$n+1$,%
\item for any~subgroup $G^{\prime}\subset\GL_{n+1}(\mathbb{C})$
such that $G^{\prime}$ does not contain reflections and
$\phi(G^{\prime})=\bar{G}$, the~singularity $\C^{n+1}/G$ is
exceptional if and only if the~singularity $\mathbb{C}^{n+1}\slash G^{\prime}$ is exceptional.%
\end{itemize}
\end{corollary}

The assumption that $G$ does not contain reflections is crucial
for Theorem~\ref{theorem:criterion} (see
\cite[Example~1.18]{ChSh09}). On the other hand, it follows from
\cite[Proposition~2.1]{Pr00} that $G$ must be primitive (see for
example \cite[Definition~1.21]{ChSh09}) if $\C^{n+1}/G$ is
exceptional. Moreover, for small $n\leqslant 4$, we have the
following

\begin{theorem}[{\cite[Theorem~1.2]{MarPr99}, \cite[Theorem~1.22]{ChSh09}}]
\label{theorem:Vanya-Kostya-invariants}  Let $G$ be a finite
subgroup in $\GL_{n+1}(\C)$ that does not contain reflections.
Suppose that $n\le 4$. Then the~following conditions are
equivalent:
\begin{itemize}
\item the~singularity $\C^{n+1}/G$ is exceptional,%
\item $\lct(\P^n, \bar{G})\ge (n+2)/(n+1)$,%
\item the~group $G$ is primitive and has no semi-invariants of degree at most $n+1$.%
\end{itemize}
\end{theorem}

In particular, both Questions~\ref{question:Yanir} and
\ref{question:exceptional} are equivalent for $n\leqslant 4$ and
can be expressed in terms of primitivity and absence of
semi-invariants of small degree of the group $G$. It appears that
in higher dimensions the latter is no longer true, since there are
non-exceptional six-dimensional quotient singularities arising
from primitive subgroups without reflections in $\GL_6(\C)$ that
have no semi-invariants of degree at most~$6$ (see
\cite[Example~3.25]{ChSh09}). On the other hand, we still believe
that both Questions~\ref{question:Yanir} and
\ref{question:exceptional} are equivalent for every $n$, which can
be summarized as

\begin{conjecture}
\label{conjecture:Vanya-Kostya-I} Let $G$ be a finite subgroup in
$\GL_{n+1}(\C)$ that does not contain reflections. Then
the~singularity $\mathbb{C}^{n+1}/G$ is exceptional if and only if
$\mathrm{lct}(\mathbb{P}^{n},\bar{G})>1$.
\end{conjecture}

In fact, Conjecture~\ref{conjecture:Vanya-Kostya-I} still holds
for $n=5$, because of

\begin{theorem}[{\cite[Theorem~1.14]{ChSh10}}]
\label{theorem:dime-6} Let $G$ be a finite subgroup in
$\SL_6(\mathbb{C})$. Then the~following are equivalent:
\begin{itemize}
\item the~singularity $\mathbb{C}^{6}\slash G$ is exceptional,%
\item the~inequality $\mathrm{lct}(\mathbb{P}^{5},\bar{G})\geqslant 7/6$ holds, %
\item either $\bar{G}$ is the~Hall--Janko sporadic simple group (see \cite{Li70}), or $G\cong 6.\A_{7}$ and $\bar{G}\cong\A_{7}$.%
\end{itemize}
\end{theorem}

In particular, both Questions~\ref{question:Yanir} and
\ref{question:exceptional} are equivalent and both have positive
answers for $n=5$. For $n=6$, both Questions~\ref{question:Yanir}
and \ref{question:exceptional} are also equivalent and both have
negative answers due to

\begin{theorem}[{\cite[Theorem~1.16]{ChSh10}}]
\label{theorem:dim-7} For every finite subgroup  $G$ in
$\GL_{7}(\C)$, the~singularity $\C^{7}/G$ is not exceptional and
$\lct(\P^n, \bar{G})\leqslant 1$.
\end{theorem}

To apply Theorem~\ref{theorem:criterion} we may assume that
$G\subset\SL_{n+1}(\mathbb{C})$, since there exists a~finite
subgroup $G^{\prime}\subset\SL_{n+1}(\mathbb{C})$ such~that
$\phi(G^{\prime})=\bar{G}$. On the other hand, it is well known
that there are at most finitely many primitive finite subgroups in
$\SL_{n+1}(\mathbb{C})$ up to conjugation by Jordan's theorem for
complex linear groups. Primitive finite subgroups of
$\SL_2(\mathbb{C})$ are group-theoretic counterparts of Platonic
solids and each of them gives rise to an exceptional quotient
singularity (see \cite[Example~5.2.3]{Sho93}). Similar
classification is possible in small dimensions. For example,
primitive finite subgroups of $\SL_{n+1}(\mathbb{C})$ for
$n\leqslant 6$ have been classified long time ago (see for example
\cite{Fe71}). This allowed to obtain the complete list of all
finite subgroups in $\SL_{n+1}(\mathbb{C})$ for every $n\leqslant
6$ that give rise to exceptional quotient singularities (see
\cite{MarPr99}, \cite{ChSh09}, and \cite{ChSh10}), which implies,
in particular, that the answers to both
Questions~\ref{question:Yanir} and \ref{question:exceptional} are
positive for every $n\leqslant 5$ and are negative for $n=6$ (see
Theorem~\ref{theorem:dim-7}). We have no idea right now what are
the answers to Questions~\ref{question:Yanir} and
\ref{question:exceptional} in the cases when $n=7$ and $n\geqslant
9$, but we expect that the answers to both
Questions~\ref{question:Yanir} and \ref{question:exceptional} may
still be negative for all $n\gg 0$ due to the following

\begin{theorem}
\label{theorem:sporadic}  Let $G$ be the finite subgroup in
$\GL_{n+1}(\C)$ such that $\bar{G}$ is a sporadic simple group.
Then $\C^{n+1}/G$ is exceptional if and only if $n=5$ and
$\bar{G}$ is the~Hall--Janko sporadic simple group.
\end{theorem}

\begin{proof}
Since $\bar{G}$ is simple, we may assume that $G$ has no
quasi-reflections. Explicit computations in GAP (see \cite{GAP})
imply that $G$ has a semi-invariant of degree at most $n+1$ (and
thus $\C^{n+1}/G$ is not exceptional by
Theorem~\ref{theorem:criterion}) unless $n=5$ and $\bar{G}$ is
the~Hall--Janko sporadic simple group\footnote{Similarly, one can
show that $G$ has a semi-invariant of degree at most $n$ (and thus
$\C^{n+1}/G$ is not weakly-exceptional (see
\cite[Definition~3.7]{ChSh09}) by  \cite[Theorem~3.16]{ChSh09})
unless either $n=5$ and $\bar{G}$ is the~Hall--Janko sporadic
simple group, or $n=11$ and $\bar{G}$ is the Suzuki sporadic
simple group (see \cite{Suzuki69}). We expect that in the latter
case the corresponding quotient singularity is actually
weakly-exceptional.}. If $n=5$ and $\bar{G}$ is the~Hall--Janko
sporadic simple group, then the singularity $\C^{n+1}/G$ is
exceptional by \cite[Theorem~1.14]{ChSh10}.
\end{proof}

An indirect evidence that both Questions~\ref{question:Yanir} and
\ref{question:exceptional} may have negative answers for all $n\gg
0$ is given by

\begin{theorem}[{\cite{Col08}}]
\label{theorem:Collins}  Let $G$ be the finite primitive subgroup
in $\GL_{n+1}(\C)$. Suppose that $n\geqslant 12$. Then
$|\bar{G}|\leqslant (n+2)!$. Moreover, if $|\bar{G}|=(n+2)!$, then
$\bar{G}\cong\SS_{n+2}$.
\end{theorem}

In fact, Collins obtained the optimal bounds for $|\bar{G}|$ for
every $n\leqslant 11$ if $G$ is primitive (his proof uses known
lower bounds for the degrees of the faithful representations of
each quasisimple group, for which the classification of finite
simple groups is required). Moreover, it follows from
\cite{Sho93}, \cite{MarPr99}, \cite{Col08}, \cite{ChSh09}, and
\cite{ChSh10} that $|\bar{G}|$ reaches its maximum on a subgroup
$\bar{G}$ in $\mathrm{Aut}(\P^n)$ with
$\mathrm{lct}(\P^n,\bar{G})>1$ if $n\leqslant 3$, and this is no
longer true for $4\leqslant n\leqslant 6$. Surprisingly, it
follows from Theorem~\ref{theorem:main} that in the case when
$n=8$, the number $|\bar{G}|$ reaches its maximum if $G$ is
isoclinic to a finite subgroup in $\GL_{9}(\C)$ that is mentioned
in Theorem~\ref{theorem:main}. For $n=11$, the number $|\bar{G}|$
reaches its maximum if $\bar{G}$ is the Suzuki sporadic simple
group.

\smallskip

In the remaining part of the paper we prove
Theorem~\ref{theorem:main}. Let $G$ be a~finite subgroup in
$\SL_9(\mathbb{C})$ from Theorem~\ref{theorem:main}. Then the
embedding $G\hookrightarrow\SL_9(\mathbb{C})$ is given by an
irreducible nine-dimensional $G$-representation\footnote{Note that
the group $G$ has two irreducible representations of dimension
$9$, but they differ only by an outer automorphism of $G$, so that
the \emph{subgroup} $G\subset\SL_9(\C)$ is defined uniquely up to
conjugation.}, which we denote by $U$.

The outline of the proof of Theorem~\ref{theorem:main} is as
follows. We assume that $\lct(\mathbb{P}^{8}, {\bar{G}})<10/9$ and
seek for a contradiction. There exists a $\bar{G}$-invariant
$\mathbb{Q}$-divisor $D$ on $\mathbb{P}^{8}$ and a positive
rational number $\lambda<10/9$ such that
$D\sim_{\mathbb{Q}}-K_{\mathbb{P}^{8}}$ and the log pair
$(\mathbb{P}^{8},\lambda D)$ is strictly log canonical, i.e. log
canonical and not Kawamata log terminal. Arguing as in
\cite{ChSh09} and \cite{ChSh10}, we apply Nadel--Shokurov
vanishing (see \cite[Theorem~9.4.8]{La04}) and Kawamata
subadjunction (see \cite[Theorem~1]{Kaw98}) to obtain restrictions
on the Hilbert polynomial of the minimal center of log canonical
singularities of the log pair $(\mathbb{P}^{8},\lambda D)$ (see
\cite[Definition~1.3]{Kaw97}, \cite{Kaw98}). Composing the latter
with results coming from representation theory we obtain a
contradiction. One of the few new ingredients of the proof is the
\emph{binomial trick} (see Lemma~\ref{lemma:qivj}).

Let us list without proofs some properties
 of the $G$-representation
$U$ (Lemmas~\ref{lemma:no-semiinvariants}, \ref{lemma:splitting},
\ref{lemma:9x45}, and~\ref{lemma:9x24}) that can be verified by
direct computations. We used GAP (see \cite{GAP}) to carry them
out.

\begin{lemma}\label{lemma:no-semiinvariants}
The group $G$ does not have semi-invariants of degree $d\leqslant
11$, and there exists a semi-invariant of the group $G$ of degree
$12$.
\end{lemma}

Denote by $\Delta_k$ the collection of dimensions of irreducible
subrepresentations of $\Sym^k(U^{\vee})$. We will use the
following notation: writing $\Delta_k=[\ldots, r\times m,\ldots]$,
we mean that among the irreducible subrepresentations of
$\Sym^k(U^{\vee})$ there are exactly $r$ subrepresentations of
dimension~$m$ (not necessarily isomorphic to each other).
Furthermore, denote by $\Sigma_k$ the set of partial sums
of~$\Delta_k$, i.\,e. the set of all numbers $s=\sum r'_im_i$,
where $\Delta_k=[\ldots,r_1\times m_1,\ldots,r_2\times m_2,\ldots,
r_i\times m_i,\ldots]$ and $0\le r_i'\le r_i$ for all $i$. We use
the abbreviation $m_{i}$ for $1\times m_{i}$.

\begin{lemma}\label{lemma:splitting}
The representation $\Sym^2(U^{\vee})$ is irreducible (and has
dimension $45$). Futhermore, $\Delta_3=[5,160]$,
$\Delta_4=[45,180,270]$, $\Delta_5=[36,90,135,216,270,540]$,
$$
\Delta_6=[4,15,24,2\times 80,3\times 240,3\times 480,640],$$
$$\Delta_7=[9,36, 3\times 135, 3\times 180, 3\times 216,270,324,405,3\times 540,720, 2\times 729],$$
$$\Delta_8=[36,4\times 45,5\times 180,5\times 270,2\times 324,6\times 360,
3\times 405,7\times 540,2\times 576,720,729],$$
$$\Delta_9=[3\times 5,3\times 20,3\times 30,40,45,60,80,
12\times 160,12\times 240,10\times 480, 10\times 640,11\times
720].$$
\end{lemma}

\begin{lemma}\label{lemma:9x45}
Let $U_{45}\subset\Sym^4(U^{\vee})$ be the $45$-dimensional
irreducible subrepresentation. Then
$$U^{\vee}\otimes U_{45}\cong U_{90}\oplus U_{135}\oplus U_{180}$$
as a $G$-representation, where $U_{90}$, $U_{135}$ and $U_{180}$
are irreducible $G$-representations of dimensions $90$, $135$ and
$180$, respectively.
\end{lemma}

\begin{lemma}\label{lemma:9x24}
Let $U_{24}\subset\Sym^6(U^{\vee})$ be the $24$-dimensional
irreducible subrepresentation. Then $U^{\vee}\otimes U_{24}$ is an
irreducible $G$-representation.
\end{lemma}

Now we are ready to prove Theorem~\ref{theorem:main}. It follows
from Theorems~\ref{theorem:alpha} and \ref{theorem:criterion} that
to prove Theorem~\ref{theorem:main} it is enough to prove that
$4/3\geqslant\mathrm{lct}(\mathbb{P}^{8},\bar{G})\geqslant 10/9$.
On the other hand, one has
$\mathrm{lct}(\mathbb{P}^{8},\bar{G})\leqslant 4/3$ by
Lemma~\ref{lemma:no-semiinvariants}. In fact, we believe that
$\mathrm{lct}(\mathbb{P}^{8},\bar{G})=4/3$, but we are unable to
prove this now. To complete the proof of
Theorem~\ref{theorem:main}, we must prove that
$\mathrm{lct}(\mathbb{P}^{8},\bar{G})\geqslant 10/9$. Suppose
that~$\mathrm{lct}(\mathbb{P}^{8},\bar{G})<10/9$. Then there is
an~effective $\bar{G}$-invariant $\mathbb{Q}$-divisor
$D\sim_{\mathbb{Q}}\mathcal{O}_{\mathbb{P}^{8}}(9)$, and there is
a~positive rational number $\lambda<10/9$ such that
$(\mathbb{P}^{8},\lambda D)$ is strictly log canonical.

Let~$S$ be a~minimal center of log canonical singularities of the
log pair $(\mathbb{P}^{8},\lambda D)$ (see
\cite[Definition~1.3]{Kaw97}, \cite{Kaw98}), let $V$ be
the~$\bar{G}$-orbit of the~subvariety $S\subset\mathbb{P}^{8}$,
and let $r$ be the~number of irreducible components of
the~subvariety $V$. Then $\mathrm{deg}(V)=r\mathrm{deg}(S)$.

Arguing as in the~proofs of \cite[Theorem~1.10]{Kaw97} and
\cite[Theorem~1]{Kaw98}, we may assume that the only log canonical
centers of the log pair $(\mathbb{P}^{8},\lambda D)$ are
components of the subvariety $V$ (see \cite[Lemma~2.8]{ChSh09}).
Then it follows from \cite[Proposition~1.5]{Kaw97} that the
components of the subvariety $V$ are disjoint. In particular, if
$\mathrm{dim}(V)\geqslant 4$, then $r=1$. Put
$n=\dim(V)=\mathrm{dim}(S)$. Then $n\ne 7$ by
Lemma~\ref{lemma:no-semiinvariants}.

Let $\mathcal{I}_{V}$ be the~ideal sheaf of the subvariety
$V\subset\P^8$, and let $\Lambda$ be a general hyperplane in
$\P^8$. Put $H=\Lambda\vert_{V}$, $h_m=h^0(\OO_V(mH))$ and
$q_m=h^0(\OO_{\P^8}(m)\otimes\mathcal{I}_{V})$ for every
$m\in\mathbb{Z}$. It follows from the Shokurov--Nadel vanishing
theorem (see \cite[Theorem~9.4.8]{La04}) that
$$
q_m=h^0\Big(\OO_{\P^8}\big(m\big)\Big)-h_m={8+m\choose m}-h_m
$$
for every $m\geqslant 1$. In particular, if $n=0$, then
$r=h_1\leqslant 9$, which is impossible by
Lemma~\ref{lemma:no-semiinvariants}. Hence, we see that
$1\leqslant n\leqslant 6$.

It follows from \cite[Theorem~1]{Kaw98} that the~variety $V$ is
normal and has at most rational singularities. Moreover, it
follows from \cite[Theorem~1]{Kaw98} that for every positive
rational number $\epsilon>0$ there is an~effective
$\mathbb{Q}$-divisor $B_{V}$ on the~variety $V$ such that
$$
\Big(K_{\P^n}+\lambda D+\epsilon \Lambda\Big)\Big\vert_{V}\sim_{\mathbb{Q}} K_{V}+B_{V},%
$$
and $(V,B_{V})$ has Kawamata log terminal singularities. In
particular, taking $\epsilon$ sufficiently small, we may assume
that $K_{V}+B_{V}\sim_{\mathbb{Q}} (1-\nu)H$ for some positive
$\nu\in\mathbb{Q}$, because $\lambda<10/9$.

\begin{remark}
\label{remark:mod-9} One can show that any irreducible
representation $W$ of the group $G$ such that the center
$Z(G)\cong\Z_3$ acts non-trivially on $W$ has dimension $\dim(W)$
divisible by $9$. Therefore one has $h_i\equiv 0\mod 9$ for every
$i$ not divisible by $3$, since $Z(G)$ acts nontrivially on
$\Sym^i(U^{\vee})$ if $i$ is not divisible by $3$.
\end{remark}

\begin{remark}\label{remark:q1-q2}
One has $q_1=0$ since $U^{\vee}$ is irreducible and $q_2=0$ by
Lemma~\ref{lemma:splitting}.
\end{remark}

Put $d=H^n=\deg(V)$ and $H_V(m)=\chi(\OO_{V}(m))$. Then it follows
from the Shokurov--Nadel vanishing theorem (see
\cite[Theorem~9.4.8]{La04}) that $H_V(m)=h_m$ for every
$m\geqslant 1$. Recall that $H_V(m)$ is a Hilbert polynomial of
the subvariety $V$, which is a polynomial in $m$ of degree $n$
with leading coefficient $d/n!$.

\begin{lemma}
\label{lemma:binomial} For any non-negative integer $\delta$ one
has
$$
d=h_{\delta+n+1}-{n\choose 1}h_{\delta+n}+ {n\choose 2}h_{\delta+n-1}+\ldots+ (-1)^{n}h_{\delta+1}.%
$$
\end{lemma}

\begin{proof}
Induction by $n$.
\end{proof}

\begin{lemma}
\label{lemma:divisibility} If $4\le n\le 5$, then $d$ is divisible
by $3$. If $n=6$, then $d$ is divisible by $9$.
\end{lemma}

\begin{proof}
Suppose that $n=6$. Applying Lemma~\ref{lemma:binomial} with
$\delta=0$ and $\delta=1$, we get
\begin{equation}
\label{eq:binomial-6}
h_7-6h_6+15h_5-20h_4+15h_3-6h_2+h_1=d=h_8-6h_7+15h_6-20h_5+15h_4-6h_3+h_2,
\end{equation}
which gives $21h_6-21h_3\equiv 0 \mod 9$ by subtracting two
equalities in $(\ref{eq:binomial-6})$, reducing everything modulo
$9$, and using Remark~\ref{remark:mod-9}. Thus $h_6-h_3\equiv 0
\mod 3$, and $15h_6-6h_3\equiv 0 \mod 9$. The latter equality
combined with~$(\ref{eq:binomial-6})$ implies that $d\equiv 0 \mod
9$.

If $n=5$, a similar argument shows that $d\equiv 0 \mod 3$.

Finally, suppose that $n=4$. Applying Lemma~\ref{lemma:binomial}
for $\delta=0$, we get $h_5-4h_4+6h_3-4h_2+h_1=d$, which gives
$d\equiv 6h_3\equiv 0 \mod 3$, since $h_5$, $h_4$, $h_2$, and
$h_1$ are divisible by $3$ by Remark~\ref{remark:mod-9}.
\end{proof}

\begin{remark}\label{remark:dim-6-small-q}
If $n=6$, then $q_3=0$. Indeed, if $q_3>0$, then $q_3>2$ by
Lemma~\ref{lemma:splitting}, so that $d<9$, which is impossible by
Lemma~\ref{lemma:divisibility}. Similarly, if $n=6$ and $q_4>0$,
one has $d=9$
\end{remark}

\begin{remark}\label{remark:increase}
One has $h_m\le h_{m+1}$ and $q_m\le q_{m+1}$ for all $m\geqslant
1$.
\end{remark}

Let $\Lambda_{1},\Lambda_{2},\ldots,\Lambda_{n}$ be general
hyperplanes in $\P^8$. Put $\Pi_{j}=\Lambda_{1}\cap\ldots\cap
\Lambda_{j}$, $V_{j}=V\cap\Pi_{j}$, $H_{j}=V_{j}\cap H$, and
$B_{V_{j}}=B_{V}\vert_{V_{j}}$ for every $j\in\{1,\ldots,n\}$. Put
$V_{0}=V$, $B_{V_{0}}=B_{V}$, $H_{0}=H$, $\Pi_0=\P^8$. For every
$j\in\{0,1,\ldots,n\}$, let $\mathcal{I}_{V_{j}}$ be the ideal
sheaf of the subvariety $V_{j}\subset\Pi_{j}$. Recall that
$\Pi_{j}\cong\P^{8-j}$ and put
$q_i(V_j)=h^0(\mathcal{O}_{\Pi_{j}}(i)\otimes\mathcal{I}_{V_{j}})$
for every $j\in\{0,1,\ldots,n\}$.

\begin{lemma}\label{lemma:qivj}
Suppose that $i\geqslant j+1$ and $j\in\{1,\ldots,n\}$. Then
$$
q_i(V_j)=q_i-{j\choose 1}q_{i-1}+{j\choose 2}q_{i-2}-\ldots+(-1)^jq_{i-j}.%
$$
\end{lemma}

\begin{proof}
For every $j\in\{0,1,\ldots,n\}$, it follows from the adjunction
formula that
$$
K_{V_{j}}+B_{V_{j}}\sim_{\mathbb{Q}} (j+1-\nu)H_{j},%
$$
because  $K_{V}+B_{V}\sim_{\mathbb{Q}} (1-\nu)H$ and $(V_{j},
B_{V_{j}})$ has at most Kawamata log terminal singularities.
Applying the Nadel--Shokurov vanishing theorem to the log pair
$(V_{j}, B_{V_{j}})$, we see that
$h^{1}(\mathcal{O}_{V_{j}}(i))=0$ for every $i\geqslant j+1$ and
every $j\in\{0,1,\ldots,n\}$. Thus, we have
\begin{equation}
\label{equation:vanishin-1}
h^{0}\Big(\mathcal{O}_{V_{j}}\big((i+1)H_{j}\big)\Big)-h^{0}\Big(\mathcal{O}_{V_{j}}\big(iH_{j}\big)\Big)=h^{0}\Big(\mathcal{O}_{V_{j+1}}\big((i+1)H_{j+1}\big)\Big)
\end{equation}
for every $i\geqslant j+1$ and every $j\in\{0,\ldots,n-1\}$. Now
applying the Nadel--Shokurov vanishing theorem to the log pair
$(\Pi_{j},\lambda D\vert_{\Pi_{j}})$, we see that
$h^{1}(\mathcal{O}_{\Pi_{j}}(i)\otimes\mathcal{I}_{V_{j}})=0$ for
every $i\geqslant j+1$ and every $j\in\{0,1,\ldots,n\}$. This
implies that
\begin{equation}
\label{equation:vanishin-2}
q_i(V_{j})={8-j+i\choose i}-h^{0}\Big(\mathcal{O}_{V_{j}}\big(iH_{j}\big)\Big)%
\end{equation}
for every $i\geqslant j+1$ and every $j\in\{0,1,\ldots,n\}$.
Combining $(\ref{equation:vanishin-1})$ and
$(\ref{equation:vanishin-2})$, we have
$$
q_{i}(V_{j-1})-q_{i-1}(V_{j-1})={9-j\choose i}-h^{0}\Big(\mathcal{O}_{V_{j-1}}\big(iH_{j-1}\big)\Big)-{8-j+i\choose i-1}+h^{0}\Big(\mathcal{O}_{V_{j-1}}\big((i-1)H_{j-1}\big)\Big)=%
$$
$$
={9-j\choose i}-{8-j+i\choose i-1}-h^{0}\Big(\mathcal{O}_{V_{j}}\big(iH_{j}\big)\Big)={8-j+i\choose i}-h^{0}\Big(\mathcal{O}_{V_{j}}\big(iH_{j}\big)\Big)=q_{i}(V_{j})%
$$
for every $i\geqslant j+1$ and every $j\in\{1,\ldots,n\}$. Thus,
we see that
\begin{equation}
\label{equation:binom}
q_{i}(V_{j})=q_{i}(V_{j-1})-q_{i-1}(V_{j-1})
\end{equation}
for every $i\geqslant j+1$ and every $j\in\{1,\ldots,n\}$.
Iterating $(\ref{equation:binom})$, we obtain the required
equality.
\end{proof}

Lemma~\ref{lemma:qivj} allows one to obtain bounds on the numbers
$q_i$.

\begin{remark}\label{remark:qivj-stupid}
There are trivial bounds $0\le q_i(V_j)<{8-j+i\choose i}$.
\end{remark}

Recall that $q_1=q_2=0$ by Remark~\ref{remark:q1-q2}, and $q_3=0$
if $n=6$ by Remark~\ref{remark:dim-6-small-q}. Therefore,
Remark~\ref{remark:qivj-stupid} implies

\begin{corollary}\label{corollary:qivj-stupid}
If $n=5$, one has
\begin{equation}\label{eq:stupid-5}
0\le q_4-3q_3<126.
\end{equation}
If $n=6$, one has
\begin{equation}\label{eq:stupid-6}
0\le q_5-4q_4<126.
\end{equation}
\end{corollary}

Playing with the numbers $q_i(V_j)$, we can obtain

\begin{lemma}\label{lemma:qivj-curve}
Suppose that $n\ge 4$. Then
$${9\choose n}-\frac{nd}{2}>q_n(V_{n-1})\geqslant {9\choose n}-nd-1.$$
\end{lemma}

\begin{proof}
Recall that the variety $V_{n-1}\subset\P^{8-n+1}$ is a smooth
curve of degree $d$, since $V$ is normal. Since $n\geqslant 4$, we
see that $V_{n-1}$ is irreducible. Let $g$ be the genus of the
curve $V_{n-1}$. It follows from the adjunction formula that
$K_{V_{n-1}}+B_{V_{n-1}}\sim_{\mathbb{Q}} (n-\nu)H_{n-1}$, because
$K_{V}+B_{V}\sim_{\mathbb{Q}} (1-\nu)H$. In particular, one has
$2g-2<dn$.

Applying the Nadel--Shokurov vanishing theorem to the log pair
$(\Pi_{n-1},\lambda D\vert_{\Pi_{n-1}})$, we see that
$$
q_m(V_{n-1})={8-n+1+m\choose
m}-h^{0}\Big(\mathcal{O}_{V_{n-1}}\big(mH_{n-1}\big)\Big)%
$$
for every $m\geqslant n$. Since $2g-2<dn$, the divisor $nH_{n-1}$
is non-special. Therefore, it follows from the Riemann--Roch
theorem that
$$
q_n(V_{n-1})={9\choose n}-nd+g-1,%
$$
which implies the required inequalities, since $2g-2<dn$ and
$g\geqslant 0$.
\end{proof}

Combining Lemmas~\ref{lemma:qivj-curve} and~\ref{lemma:qivj} and
recalling the trivial bounds from Remark~\ref{remark:qivj-stupid},
we obtain

\begin{corollary}\label{corollary:curve}
If $n=4$, then
\begin{equation}\label{eq:curve-4}
\max\big(0, 125-4d\big)\le q_4-3q_3\le 125-2d.
\end{equation}
If $n=5$, then
\begin{equation}\label{eq:curve-5}
\max\big(0, 125-5d\big)\le q_5-4q_4+6q_3\le 126-\frac{5d}{2}.
\end{equation}
If $n=6$, then
\begin{equation}\label{eq:curve-6}
0\le q_6-5q_5+10q_4-10q_3\le 83-3d.
\end{equation}
\end{corollary}

As a by-product of Corollary~\ref{corollary:curve}, we get

\begin{corollary}\label{corollary:degree-bounds}
If $n=4$, then $d\le 62$. If $n=5$, then $d\le 50$. If $n=6$, then
$d\le 27$.
\end{corollary}

The above restrictions reduce the problem to a combinatorial
question of finding all polynomials $H_V$ of degree $n$ with a
leading coefficient $d/n!$, such that $h_m=H_V(m)\in\Sigma_m$ for
sufficiently many $m\geqslant 1$, and such that the numbers $h_m$
and $q_m=h^0(\OO_{\P^8}(m))-h_m$ satisfy the conditions arising
from Lemma~\ref{lemma:divisibility},
Corollaries~\ref{corollary:qivj-stupid}, \ref{corollary:curve}
and~\ref{corollary:degree-bounds}, and Remarks~\ref{remark:q1-q2},
\ref{remark:dim-6-small-q} and~\ref{remark:increase}. This can be
done in a straighforward way, although the number of cases to be
considered is so large that we had to delegate this part of the
proof to a simple computer program. Finally, we get the following
four lemmas which we leave without proofs.

\begin{lemma}\label{lemma:python-3}
There are no polynomials $\mathrm{H}(m)$ of degree $n\le 3$ such
that the values $h_m=\mathrm{H}(m)$ are in $\Sigma_m$ for $1\le
m\le 6$ and $h_i\le h_{i+1}$ for $1\le i\le 5$.
\end{lemma}

\begin{lemma}\label{lemma:python-4}
If $\mathrm{H}(m)$ is a polynomial of degree $n=4$ with a leading
coefficient $d/n!$ with $d\leqslant 62$ and $d$ divisible by $3$,
such that the values $h_m=\mathrm{H}(m)$ are in $\Sigma_m$ for
$1\le m\le 6$, and the numbers $h_m$ and $q_m={8+m\choose m}-h_m$
satisfy the bounds of Remark~\ref{remark:increase} and
$(\ref{eq:curve-4})$, then $d=36$, $q_1=q_2=q_3=0$, $q_4=45$,
$q_5=270$.
\end{lemma}

\begin{lemma}\label{lemma:python-5}
If $\mathrm{H}(m)$ is a polynomial of degree $n=5$ with a leading
coefficient $d/n!$ with $d\leqslant 50$ and $d$ divisible by $3$,
such that the values $h_m=\mathrm{H}(m)$ are in $\Sigma_m$ for
$1\le m\le 9$, and the numbers $h_m$ and $q_m={8+m\choose m}-h_m$
satisfy the bounds of Remark~\ref{remark:increase},
$(\ref{eq:curve-5})$ and~$(\ref{eq:stupid-5})$, then $d=45$,
$q_1=q_2=q_3=q_4=q_5=0$, $q_6=39$, $q_7=270$.
\end{lemma}

\begin{lemma}\label{lemma:python-6}
There are no polynomials $\mathrm{H}(m)$ of degree $n=6$ with a
leading coefficient $d/n!$ with $d\leqslant 27$ and $d$ divisible
by $9$, such that the values $h_m=\mathrm{H}(m)$ are in $\Sigma_m$
for $1\le m\le 9$, and the numbers $h_m$ and $q_m={8+m\choose
m}-h_m$ satisfy the bounds of Remark~\ref{remark:increase},
$(\ref{eq:curve-6})$ and~$(\ref{eq:stupid-6})$.
\end{lemma}

Applying Lemmas~\ref{lemma:python-3}, \ref{lemma:python-4},
\ref{lemma:python-5}, and \ref{lemma:python-6} to the Hilbert
polynomial $H_{V}(m)$, we end up with the following two
possibilities: either $n=4$, $d=36$, $q_1=q_2=q_3=0$, $q_{4}=45$,
$q_5=270$, or $n=5$, $d=45$, $q_1=q_2=q_3=q_4=q_5=0$, $q_6=39$,
$q_7=270$.

\begin{remark}\label{remark:non-zero}
Let $W$ be a $G$-subrepresentation in
$H^0\big(\mathcal{I}_V\otimes\OO_{\P^{8}}(m)\big)$. Then there is
a natural map of $G$-representations $\psi\colon U^{\vee}\otimes
W\to H^0(\mathcal{I}_V\otimes\OO_{\P^{8}}(m+1))$, which is
obviously a non-zero map.
\end{remark}

Let us suppose that $n=4$. Then $q_4=45$, so that by
Lemma~\ref{lemma:splitting} there is an irreducible
$45$-dimensional $G$-subrepresentation $U_{45}\subset
H^0(\mathcal{I}_V\otimes\OO_{\P^{8}}(4))$. Thus, there is a
morphism of $G$-representations $\psi_5\colon U^{\vee}\otimes
U_{45}\to H^0(\mathcal{I}_V\otimes\OO_{\P^{8}}(5))$, which is a
non-zero map by Remark~\ref{remark:non-zero}. On the other hand,
$H^0\big(\mathcal{I}_V\otimes\OO_{\P^{8}}(5)\big)$ has no
$180$-dimensional $G$-subrepresentations by
Lemma~\ref{lemma:splitting}. Put
$q_5^{\prime}=\mathrm{dim}(\mathrm{Im}\psi_{5})$.  Keeping in mind
the splitting of $U^{\vee}\otimes U_{45}$ described in
Lemma~\ref{lemma:9x45}, we see that $q_5^{\prime}$ equals either
$90$, or $135$, or $225$. Since $q_5=270$, there must exist a
(possibly reducible) $G$-subrepresentation of
$H^0(\mathcal{I}_V\otimes\OO_{\P^{8}}(5))$ of dimension
$q_5-q_5^{\prime}$, i.\,e. of dimension $180$, $135$ and~$45$,
respectively. Neither of these cases is possible by
Lemma~\ref{lemma:splitting} (in particular, the second case is
impossible since
$H^0\big(\mathcal{I}_V\otimes\OO_{\P^{8}}(5)\big)$ contains a
unique $G$-invariant subspace of dimension $135$). Therefore, one
has $n\ne 4$.

Finally, we see that $n=5$. Then  $q_6=39$, so that by
Lemma~\ref{lemma:splitting} there is an irreducible
$24$-dimensional $G$-subrepresentation $U_{24}\subset
H^0(\mathcal{I}_V\otimes\OO_{\P^{8}}(6))$. Therefore there is a
morphism of $G$-representations $\psi_7\colon U^{\vee}\otimes
U_{24}\to H^0(\mathcal{I}_V\otimes\OO_{\P^{8}}(7))$, which is a
non-zero map by Remark~\ref{remark:non-zero}. Since
$U^{\vee}\otimes U_{24}$ is an irreducible $216$-dimensional
$G$-representation by Lemma~\ref{lemma:9x24}, we see that $\psi_7$
is injective. Since $q_7=270$, there must exist a (possibly
reducible) $G$-subrepresentation of
$H^0(\mathcal{I}_V\otimes\OO_{\P^{8}}(7))$ of dimension
$270-216=54$, which is impossible by Lemma~\ref{lemma:splitting}.
The obtained contradiction completes the proof of
Theorem~\ref{theorem:main}.

\medskip

This work was partially supported by the grants N.Sh.-4713.2010.1,
RFFI 11-01-00336-a, RFFI 11-01-92613-KO-a, RFFI 08-01-00395-a,
RFFI 11-01-00185-a, NSF DMS-1001427 and LMS grant 41007.

We would like to thank Michael Collins for a reference to
Theorem~\ref{theorem:Collins}. Special thanks goes to Andrey
Zavarnitsyn who explained us various facts from Group Theory and
provided us with the results of GAP computations used in the proof
of Theorem~\ref{theorem:sporadic}. The first author would like to
thank personally Selman Akbulut for inviting him to Eighteenth
Gokova Geometry and Topology conference (Gokova, Turkey).

\end{document}